\documentclass[11pt]{article}
\usepackage{mathtools,amsthm, amssymb,bm}
\usepackage{enumerate,comment}
\usepackage{accents}
\usepackage{caption}
\usepackage{cases}
\usepackage{xcolor}
\usepackage{thm-restate}
\usepackage{amscd}
\usepackage{latexsym,epsfig}
\usepackage{extarrows}
\usepackage[all]{xy}
\usepackage{pst-all}
\usepackage{ulem}

\makeatletter
\def\sbar{\accentset{{\cc@style\underline{\mskip16mu}}}}
\def\mbar{\accentset{{\cc@style\underline{\mskip32mu}}}}
\def\lbar{\accentset{{\cc@style\underline{\mskip48mu}}}}

\makeatother

\theoremstyle{definition}
\newtheorem{Def}{Definition}[section]

\newtheorem{Rem}[Def]{Remark}
\newtheorem{Expect}[Def]{Expectation}

\theoremstyle{plain}
\newtheorem{Thm}[Def]{Theorem}

\newtheorem{Prop}[Def]{Proposition}
\newtheorem{Lem}[Def]{Lemma}

\newtheorem*{thm}{Theorem}

\newtheorem{Fact}[Def]{Fact}

\newcommand{\an}[2][n]{\left(#2\mathrel;#1\right)}
\newcommand{\pd}[1]{\dfrac{\partial}{\partial #1}}
\newcommand{\genpdd}[2]{\dfrac{\partial^2}{\partial{#1}\partial{#2}}}
\newcommand{\pdd}[1]{\dfrac{\partial^2}{\partial{#1}^2}}

\newcommand{\fps}{[\![z_1,z_2,z_3]\!]}
\newcommand{\diff}[1]{\mathop{\mathcal{D}_{#1}}}
\newcommand{\diffchar}[2]{\mathop{\mathcal{D}_{#1}^{(#2)}}}
\newcommand{\diffsimp}[2]{\mathop{\mathcal{D}_{#1,\,#2}}}
\newcommand{\ALseries}{\mathcal{F}(a,b_1,b_2,b_3,c\,;z_1,z_2,z_3)}
\newcommand{\subALseriesdef}{%
\frac{\an[\sum_{k=1}^3 n_k]{a} \prod_{k=1}^3 \an[n_k]{b_k}}%
{\an[\sum_{k=1}^3 n_k]{c} \prod_{k=1}^3 \an[n_k]{1}}%
}
\newcommand{\ALseriesdef}{%
\sum_{n_1=0}^\infty\sum_{n_2=0}^\infty\sum_{n_3=0}^\infty\subALseriesdef%
}
\newcommand{\setsymbol}[2]{\left\{ #1 \mathrel{} \middle| \mathrel{} #2 \right\}}
\newcommand{\setsymbolin}[3]{\left\{ #1 \in #2 \mathrel{} \middle| \mathrel{} #3 \right\}}
\usepackage{xparse}
\NewDocumentCommand{\zzz}{o o o}{%
    \IfValueTF{#1}{%
        z_1^{n_1#1}z_2^{n_2#2}z_3^{n_3#3}
    }{
        \IfValueTF{#2}{%
            z_1^{n_1}z_2^{n_2#2}z_3^{n_3#3}
        }{%
            \IfValueTF{#3}{%
                z_1^{n_1}z_2^{n_2}z_3^{n_3#3}
            }{%
                z_1^{n_1}z_2^{n_2}z_3^{n_3}
            }
        }
    }
}
\NewDocumentCommand{\nnn}{o o o}{%
    \IfValueTF{#1}{%
        n_1{#1},n_2{#2},n_3{#3}
    }{%
        \IfValueTF{#2}{%
            n_1,n_2{#2},n_3{#3}
        }{%
            \IfValueTF{#3}{%
                n_1,n_2,n_3{#3}
            }{%
                n_1,n_2,n_3
            }
        }
    }
}
\NewDocumentCommand{\truncatedAL}{o o}{%
    \IfValueTF{#1}{%
        \IfValueT{#2}{%
            \Tilde{\mathcal{F}}_{#1,#2}
        }
    }{
        \Tilde{\mathcal{F}}_{i,j}
    }
}
\NewDocumentCommand{\supp}{o o}{%
    \IfValueTF{#1}{%
        \IfValueT{#2}{%
            \mathsf{supp}^{(#1,#2)}
        }
    }{
        \mathsf{supp}^{(i,j)}
    }
}
\NewDocumentCommand{\Legendresymbol}{o o}{%
    \IfValueTF{#1}{%
        \IfValueT{#2}{%
            \genfrac{(}{)}{}{}{#1}{#2}
        }
    }{
            \genfrac{(}{)}{}{}{-1}{p}
    }
}
\newcommand{\NNI}{\mathbb{Z}_{\ge 0}}

\begin{document}

\title
{\bf The multiplicity-one theorem for the superspeciality of curves of genus two}
\author{
Shushi Harashita\thanks{Graduate School of Environment and Information Sciences, Yokohama National University.
E-mail: \texttt{harasita@ynu.ac.jp}}
\ and
Yuya Yamamoto\thanks{Graduate School of Environment and Information Sciences, Yokohama National University.
E-mail: \texttt{yamamoto-yuya-cm@ynu.jp}}
}
\maketitle
\begin{abstract}
\noindent\quad 
Igusa proved in 1958 that the polynomial determining the supersingularity of elliptic curve\textcolor{black}{s} in Legendre form is separable.
In this paper, we get an analogous result for curves of genus $2$ in Rosenhain form.
More precisely we show that the ideal determining the superspeciality of the curve has multiplicity one at every superspecial point.
Igusa used a Picard-Fucks differential operator annihilating a Gau{\ss}  hypergeometric series.
We shall use \textcolor{black}{the} Lauricella system (of type D) of hypergeometric differential equations in three variables.


\end{abstract}

\section{Introduction}
Let $K$ be a field of characteristic $p>2$.
Consider an elliptic curve of Legendre form
\[
E:\quad y^2 = x(x-1)(x-t)
\]
over $K$. It is well known (cf.~\cite[p.255]{Deuring} and \cite[Theorem V.4.1]{Silverman}) that $E$ is supersingular if and only if
$H_p(t) =0$, where
\[
H_p(t) := \sum_{i=0}^{(p-1)/2}{\binom{(p-1)/2}{i}}^2
t^i.
\]
A celebrated result by Igusa \cite{Igusa} says that
the polynomial $H_p(t)$ is separable. His proof uses the fact that
$H_p(z)$ \textcolor{black}{satisfies} the differential equation
\[
\mathop{\mathcal D} {H_p(z)}= 0
\]
with
\begin{equation}\label{PDE-GenusOne}
\mathop{\mathcal D} = z(1-z) \dfrac{d\textcolor{black}{^2}}{dz^2} + (c-(a+b+1)z)\dfrac{d}{dz}-ab
\end{equation}
for $a=1/2$, $b=1/2$ and $c=1$ (cf.~Igusa~\cite{Igusa}).
For $a,b,c \in \mathbb{C}$ with $-c \notin \mathbb{Z}_{\ge 0}$,
over $\mathbb C$
the differential equation \eqref{PDE-GenusOne} is satisfied by
the Gau{\ss} hypergeometric series
\[
F(a,b,c\,;z) \coloneqq \sum_{n=0}^\infty \dfrac{\an{a}\an{b}}{\an{c}\an{1}}z^n,
\]
where $\an{x} = x(x+1)(x+2)\cdots(x+n-1)$. 

In this paper, we show an analogous result in the case of genus $2$.
For smaller families of genus-$2$ curves,
similar results have been obtained in \cite[Proposition 1.14]{IKO}.
In the present work, however, we establish such a result for an entire family of genus-$2$ curves.
Let $C$ be the normalization of the projective model of
\[
C:\quad y^2 = f(x):=x(x-1)(x-\lambda_1)(x-\lambda_2)(x-\lambda_3)
\]
with $\#\{0,1,\lambda_1,\lambda_2,\lambda_3\} = 5$.
Note that $C$ is a nonsingular curve of genus 2.
A curve $C$ is called {\it superspecial} if its Jacobian ${\rm Jac}(C)$
is isomorphic to a product of supersingular elliptic curves over an algebraically closed field. This is equivalent to that
the Cartier operator on $H^0(C,{\varOmega}_C)$ is zero (cf.~\cite[Theorem 4.1]{Nygaard}). The matrix (called Cartier-Manin matrix) representing the operator with respect to the basis $\{dx/y, xdx/y\}$ is described as
\[
\begin{pmatrix}
    c_{p-1} & c_{p-2}\\
    c_{2p-1} & c_{2p-2}
\end{pmatrix},
\]
where $c_k$ is the $x^k$-coefficient of $f(x)^{(p-1)/2}$, see \cite[p.~79]{Manin}.
We shall regard $\lambda_1,\lambda_2,\lambda_2$ as indeterminates
and consider  $c_{ip-j}$
as polynomials in $\lambda_1,\lambda_2,\lambda_2$ for $1\le i,j\le 2$.
The space $V(c_{p-1},c_{p-2},c_{2p-1},c_{2p-2})$ consisting of points
where $c_{ip-j}$ vanish for all $i,j$ with $1\leq i,j\leq 2$
is called the {\it superspecial locus}, which is known to consist of finite points.
The aim of this paper is to prove the ideal generated by
$c_{ip-j}$ for $1\le i, j\le 2$ is of multiplicity one
at all point of $V(c_{p-1},c_{p-2},c_{2p-1},c_{2p-2})$.
To achieve this we use 
\textcolor{black}{the} Lauricella hypergeometric series of type (D) in three variables and the partial differential equations they satisfy.
Lauricella hypergeometric series of type (D) in three variables is
defined as follows (cf.~Definition \ref{def:Appell}):
\begin{equation}\label{eq:L-hyp-geom-series}
\ALseries
\coloneqq 
 \sum_{n_1=0}^\infty \, \sum_{n_2=0}^\infty \, \sum_{n_3=0}^\infty A_{n_1,n_2,n_3}\zzz,
\end{equation}
where
\begin{equation}
    A_{n_1,n_2,n_3}:=\subALseriesdef
\end{equation}
with $a,b_1,b_2,b_3,c \in \mathbb{C},-c \notin {\mathbb{Z}}_{\ge 0}$.

In this paper, the hypergeometric series is considered as an element of
the ring of formal power series in $z_1, z_2$ and $z_3$.
The hypergeometric series satisfies the partial
differential equations $\diff{\ell} w=0$ for $\ell=1,2,3$ 
and $\diffsimp{k}{\ell} w=0$ for $1\le k<\ell\le 3$, where
\begin{eqnarray}
\diff{\ell}
&=&z_\ell(1-z_\ell)\pdd{z_\ell}+\sum_{\substack{1 \le k \le 3,\\k \neq \ell}}
{ z_k(1-z_\ell)\genpdd{z_\ell}{z_k} }  \notag\\
&&+\left(c-(a+b_\ell+1)z_\ell\right)\pd{z_\ell}
-\sum_{\substack{1 \le k \le 3,\\k \neq \ell}}{ b_\ell z_k \pd{z_k} }-ab_\ell
\label{eq:Dl}
\end{eqnarray}
and
\begin{equation}
\diffsimp{\ell}{m}:=
(z_\ell-z_m)\dfrac{\partial^2}{\partial z_\ell\partial z_m}
-b_m\dfrac{\partial}{\partial z_\ell}
+b_\ell\dfrac{\partial}{\partial z_m}.
\end{equation}


Let $j\in\{1,2\}$.
Let $\diffchar{\ell}{j}$ be
the above partial differential operators ${\mathcal D}_\ell$
associated to 
$a = 5/2 - j$, $(b_1,b_2,b_3)=(1/2,1/2,1/2)$ and $c = 3-j$.
The first main theorem is
\begin{restatable}{theor}{thmA}\label{thm:PDEinPositiveChar}
For every $i,j$ with $1\le i,j \le 2$, the entry $c_{ip-j}$ of the Cartier-Manin matrix
satisfies
\begin{equation*}
\diffchar{\ell}{j} c_{ip-j} = 0
\end{equation*}
for $\ell=1,2,3$ and
\begin{equation*}
\diffsimp{\ell}{m} c_{ip-j} = 0
\end{equation*}
for $1\le \ell < m \le 3$.
\end{restatable} 

As an application of this theorem and the contiguity relations obtained in Subsection \ref{subsec:Contiguity}, we have the second theorem,
which can be \textcolor{black}{regarded} as an analogous result in the genus-two case of Igusa's result \cite{Igusa} in the genus-one case.
We call this theorem {\it the multiplicity-one theorem for the superspeciality of the curve $y^2=x(x-1)(x-z_1)(x-z_2)(x-z_3)$.}

\begin{restatable}{theor}{thmB}\label{thm:MultiplicityOne}
The scheme defined by the ideal generated by
\[
c_{p-1}, c_{2p-1}, c_{p-2}, c_{2p-2}
\]
of ${\mathbb F}_p[z_1,z_2,z_3]$ is reduced.
(This is equivalent to saying that the scheme is nonsingular.
It is well-known that the scheme is zero-dimensional.)
\end{restatable}

This paper is organized as follows.
In Section 2, we review some facts, which will be used in the \textcolor{black}{later} sections.
In Section 3, we restate the main theorems.
In Section 4, we prove the first main theorem. 
In Section 5, we prove the second main theorem. For this, we prove
some relations of
entries of Cartier-Manin matrix
with some partial differential operators.
This can be called contiguity relations.
This was another key ingredient of the proof of the second main theorem.


\subsection*{Acknowledgements}
\textcolor{black}{The authors thank the referee for his/her helpful suggestions and comments.}
This paper is a part of master's thesis of the second author. He thanks the first author for his supervision and much helpful advice.
The authors are grateful to Ryo Ohashi for his contribution to the proof of
Proposition \ref{prop:SspImpliesNonsing},
and also to Takumi Ogasawara and Yuki Nakada for their important comments.
This work was supported by JSPS Grant-in-Aid for Scientific Research\,(C)\ 21K03159.

\section{Preliminaries}
In this section, we recall \textcolor{black}{the} Lauricella hypergeometric series
and partial differential equations they satisfy.
We also review the relation between the Cartier-Manin matrix of a hyperelliptic curve and the associated 
 Lauricella hypergeometric series. Finally we discuss
 the superspeciality and the nonsingularity of curves of genus two.

\subsection{Lauricella hypergeometric series}
\textcolor{black}{The} Lauricella hypergeometric series in three variables is defined as follows:
\begin{Def}\label{def:Appell}
Consider an element $\ALseries$ of $\mathbb{C}\fps$:
\begin{equation}
\ALseriesdef \zzz, \label{eq:Appell3}
\end{equation}
where $a,b_1,b_2,b_3,c \in \mathbb{C},-c \notin {\mathbb{Z}}_{\ge 0}$.
We call it \textcolor{black}{the} Lauricella hypergeometric series in three variables
with respect to $(a,b_1,b_2,b_3,c)$.
\end{Def}

\textcolor{black}{The} Lauricella hypergeometric series \textcolor{black}{satisfies} the following partial differential equations,
see $E_D(a,(b),c)$ in \cite[Table 3.2 in Section 3.5]{Matsumoto}.

\begin{Thm}\label{thm:main1}
\textcolor{black}{The} Lauricella hypergeometric series in three variables $w\coloneqq\ALseries$ satisfies
\begin{equation}\label{eq:main}
    \diff{\ell}w=0
\end{equation}
for $\ell = 1,2,3$ and
\begin{equation}\label{eq:main2}
    \diffsimp{\ell}{m}w=0
\end{equation}
for $1\le \ell < m\le 3$,
where the partial differential operators $\diff{1}$, $\diff{2}$, $\diff{3}$ are given by
\begin{align*}
\diff{1}
=&z_1(1-z_1)\pdd{z_1}+z_2(1-z_1)\genpdd{z_1}{z_2}
+z_3(1-z_1)\genpdd{z_1}{z_3}\\
&+\left(c-(a+b_1+1)z_1\right)\pd{z_1}-b_1 z_2\pd{z_2}-b_1 z_3\pd{z_3}-a b_1,
\end{align*}
\begin{align*}
\diff{2}
=&z_1(1-z_2)\genpdd{z_1}{z_2}+z_2(1-z_2)\pdd{z_2}
+z_3(1-z_2)\genpdd{z_2}{z_3}\\
&-b_2 z_1\pd{z_1}+\left(c-(a+b_2+1)z_2\right)\pd{z_2}-b_2 z_3\pd{z_3}-a b_2,
\end{align*}
\begin{align*}
\diff{3}
=&z_1(1-z_3)\genpdd{z_1}{z_3}+z_2(1-z_3)\genpdd{z_2}{z_3}
+z_3(1-z_3)\pdd{z_3}\\
&-b_3 z_1\pd{z_1}-b_3 z_2\pd{z_2}+\left(c-(a+b_3+1)z_3\right)\pd{z_3}-a b_3,
\end{align*}
and 
\begin{equation}
\diffsimp{\ell}{m}:=
(z_\ell-z_m)\dfrac{\partial^2}{\partial z_\ell\partial z_m}
+b_\ell\dfrac{\partial}{\partial z_m}
-b_m\dfrac{\partial}{\partial z_\ell}
\end{equation}
for $1\le \ell < m \le 3$.
\end{Thm}

\begin{Rem}\label{rem:StandardRelation}
Let $K$ be any field.
The partial differential operators $\diff{\ell}$ for $\ell = 1,2,3$
can be considered if $a,b_1,b_2,b_3,c\in K$.
Let $\textcolor{black}{F}$ be any element of $K\fps$, say
\[
\textcolor{black}{F} = \sum_{\nnn\ge0} B_{\nnn}\zzz.
\]
A direct computation shows that the $\zzz$-coefficient of
$\diff{1} \textcolor{black}{F}$ is
\[
(1+n_1)(c+n_1+n_2+n_3) B_{\nnn[+1][][]}
-(b_1+n_1)(a+n_1+n_2+n_3) B_{\nnn}.
\]
If the characteristic of $K$ is $0$ and $\diff{1}\textcolor{black}{F}=0$,
then we have the recurrence relation
\begin{equation}\label{eq:StandardRelation}
B_{\nnn[+1][][]} = \frac{(b_1+n_1)(a+n_1+n_2+n_3)}{(1+n_1)(c+n_1+n_2+n_3) }B_{\nnn}.
\end{equation}
We call this the {\it standard \textcolor{black}{recurrence} relation} for $\textcolor{black}{F}$ and $\diff{1}$.
The similar things also hold for $\diff{2} \textcolor{black}{F}$ and $\diff{3} \textcolor{black}{F}$.
This shows that if the characteristic of $K$ is $0$, then the differential equations \eqref{eq:main}
determine $\textcolor{black}{F}$ from $B_{0,0,0}$ provided $B_{0,0,0}\ne 0$.
\end{Rem}

\subsection{Cartier-Manin matrices
and hypergeometric series}\label{subsec:CM-HG}
Let $p$ be an odd prime and $K$ a field of characteristic $p$.
Let $g$ be a natural number.
A hyperelliptic curve of genus $g$ is defined by $y^2 = f(x)$
with separable polynomial $f(x) \in K[x]$ of degree $2g+1$ or $2g+2$.
Let us review its Cartier-Manin matrix.

\begin{Def}\label{def:CMMatrix}
For the hyperelliptic curve $C$ defined by $y^2 = f(x)$ with
\[
f(x)=x(x-1)(x-z_1)\cdots(x-z_{2g-1}),
\]
let
\[
f(x)^{\frac{p-1}{2}}=\sum_{k=0}^{\frac{p-1}{2}(2g+1)} {c_k x^k}.
\]
be the expansion of $f(x)^{\frac{p-1}{2}}$.
The Cartier-Manin matrix of $C$ is defined to be
\[
M = { \Bigl( c_{ip-j} \Bigr) }_{i,j=1,\ldots,g}
\]
(cf.\ \cite[Section 2]{Yui}).
\end{Def}

\begin{Rem}
In this paper, we consider the case of $g=2$.
\end{Rem}

We review truncations of \textcolor{black}{the} Lauricella hypergeometric series
and a description of $c_{ip-j}$ with respect to them, due to  Ohashi and Harashita \cite{OH}.
\begin{Def}[Ohashi-Harashita~{\cite[Definition 6.4 and Example 6.10]{OH}}]\label{def:Truncation}
For $1 \le i,j \le g=2$, we put
\begin{equation*}
\begin{cases}
{\displaystyle a' \coloneqq \dfrac{2g+1}{2} - j 
= g+\dfrac{1}{2}-j = \dfrac{5}{2}-j },\\
{\displaystyle b' \coloneqq \dfrac{1}{2} },\\
{\displaystyle c' \coloneqq a'+1-\dfrac{1}{2} = g+1-j = 3-j },\\
{\displaystyle d' \coloneqq \dfrac{p-1}{2}(2g+1) - ip + j
= \dfrac{p-1}{2} \cdot 5 - ip + j.}
\end{cases}
\end{equation*}
For \textcolor{black}{the} Lauricella hypergeometric series,
\[
\mathcal{F}(a',b',b',b',c'\,;z_1,z_2,z_3)
= \sum_{n_1=0}^\infty \, \sum_{n_2=0}^\infty \, \sum_{n_3=0}^\infty
A^{(i,j)}_{n_1,n_2,n_3}
z_1^{n_1} z_2^{n_2} z_3^{n_3}
\]
with
\begin{equation}\label{eq:Aij}
A^{(i,j)}_{n_1,n_2,n_3}
:=\dfrac{ \an[\sum_{k=1}^3 n_k]{a'} \prod_{k=1}^3 \an[n_k]{b'} }
{ \an[\sum_{k=1}^3 n_k]{c'} \prod_{k=1}^3 \an[n_k]{1} },
\end{equation}
we define its truncation
$\truncatedAL$
to be  the sum of $z_1^{n_1} z_2^{n_2} z_3^{n_3}$-term 
for triples $(\nnn)$ satisfying
\begin{equation*}
\begin{cases}
{\displaystyle n_k \le \dfrac{p-1}{2} \quad (k=1,2,3)},\\
{\displaystyle d'-\dfrac{p-1}{2} \le n_1+n_2+n_3 \le d'. }
\end{cases}
\end{equation*}
We denote the set of exponents of monomials in $\truncatedAL$ by $\supp$, namely
\begin{equation}\label{eq:supp}
\supp
=
\setsymbolin{(n_1,n_2,n_3)}{\NNI^3}{\begin{array}{l}
    n_k \le \dfrac{p-1}{2} \quad (k=1,2,3),\\
    d'-\dfrac{p-1}{2} \le n_1+n_2+n_3 \le d'
\end{array}}.
\end{equation}
A priori, $\truncatedAL$ is an element of $\mathbb{Q}[z_1,z_2,z_3]$ but
we regard $\truncatedAL$
as an element of ${\mathbb{F}_p}[z_1,z_2,z_3]$,
since the denominator of the coefficient of each term of $\truncatedAL$
is coprime to $p$.
\end{Def}

The following relation between
the Cartier-Manin matrix and truncations $\truncatedAL$ of Lauricella hypergeometric series is known.
This is a special case of the result \cite[Theorem 6.5]{OH} by Ohashi and Harashita.
\begin{Fact}
Assume $p>2$.
Over $\mathbb{F}_p$, we have
\[
c_{ip-j} = \dfrac{ \an[d']{c'} }{ \an[d']{a'} } \truncatedAL
\]
for any $(i,j)$ with $1 \le i,j \le 2$.
\end{Fact}

By \cite[Example 6.10]{OH}, we know that
\begin{equation}\label{eq:Const_OH}
\dfrac{ \an[d']{c'} }{ \an[d']{a'} }  = \dfrac{(p(4-2i+1)-1)!!}{(p(4-2i+1)-2)!!}
\dfrac{(4-2j-1)!!}{(4-2j)!!},
\end{equation}
but the value has a simpler description:
\begin{Lem}\label{lem:simplification}
We have
\[    \dfrac{ \an[d']{c'} }{ \an[d']{a'} }  \equiv 
(-1)^\frac{p-1}{2}\cdot\dfrac{j}{i}\pmod{p}. \]
\end{Lem}
\begin{proof}
By \eqref{eq:Const_OH} with $0!!=(-1)!!=1$, we have
\[
    \frac{ \an[d']{c'} }{ \an[d']{a'} }
    = \frac{\left((5-2i)p-1\right)!!}{\left((5-2i)p-2\right)!!} \cdot \frac{j}{2}
\]
for $j=1,2$. 
It suffices to show
\[
   \frac{\left((5-2i)p-1\right)!!}{\left((5-2i)p-2\right)!!}
   \equiv (-1)^{\frac{p-1}{2}} \cdot \frac{2}{i} \pmod{p},
\]
namely
\begin{alignat}{2}
    \frac{(3p-1)!!}{(3p-2)!!}
    &\equiv 2 (-1)^{\frac{p-1}{2}}  &\pmod{p} \quad\text{if}\; i &= 1, \label{eq:DF3p-1}\\ 
    \frac{(p-1)!!}{(p-2)!!}
    &\equiv (-1)^{\frac{p-1}{2}}  &\pmod{p} \quad\text{if}\; i &= 2. \label{eq:DFp-1}
\end{alignat}
Since $p$ is odd,
\eqref{eq:DFp-1} follows from
\[
    (p-1)!! = (p-1)(p-3) \dotsm 2
        \equiv (-1)(-3) \dotsm (-(p-2))
        = (-1)^{\frac{p-1}{2}} (p-2)!!.
\]
Similarly \eqref{eq:DF3p-1} follows from that $(3p-1)!!$ is equal to
\begin{align*}
& (3p-1)(3p-3)\dotsm (2p+2) \cdot 2p\cdot (2p-2)(2p-4) \dotsm (p+1) \cdot (p-1)!!\\
&\equiv (p-1)!! \cdot 2p \cdot (p-2)!!\cdot (p-1)!! =(p-1)! \cdot (p-1)!! \cdot 2p
\end{align*}
and that $(3p-2)!!$ is equal to
\begin{align*}
& (3p-2)(3p-4)\dotsm (2p+1) \cdot(2p-1)(2p-3) \dotsm (p+2) \cdot p\cdot (p-2)!!\\
&\equiv (p-2)!! \cdot (p-1)!!\cdot p \cdot (p-2)!! = (p-1)! \cdot (p-2)!! \cdot p,
\end{align*}
together with \eqref{eq:DFp-1}.
\end{proof}
%
%
%

We collect recurrence relations of ${A}_{\nnn}^{(i,j)}$, which will be used later on.
\begin{Lem}\label{lem:recurrence_relations}
\begin{enumerate}
    \item[(1)] (Standard \textcolor{black}{recurrence} relation) We have
    \[{A}_{\nnn[+1][][]}^{(i,j)}=
\frac{(\frac{5}{2}-j+n_1+n_2+n_3)(\frac{1}{2}+n_1)}{(3-j+n_1+n_2+n_3)(1+n_1)}{A}_{\nnn}^{(i,j)}
\]
for $n_1,n_2,n_3\ge0$,
and similarly for $n_2, n_3$.
\item[(2)] For $n_1,n_2,n_3 \ge 0$, we have
\begin{equation}\label{eq:ALinterchanging}
{A}_{\nnn}^{(i,1)} = \frac{1+2(n_1+n_2+n_3)}{1+n_1+n_2+n_3} {A}_{\nnn}^{(i,2)}
\end{equation}
and
\begin{equation}\label{eq:ALinterchanging2}
{A}_{\nnn[+1][][]}^{(i,2)}=\frac{\frac{1}{2}+n_1}{2(1+n_1)}{A}_{\nnn}^{(i,1)}. 
\end{equation}
\end{enumerate}
\begin{proof}
(1) is a special case of the standard \textcolor{black}{recurrence} relation in Remark \ref{rem:StandardRelation}.

(2) As for \eqref{eq:ALinterchanging}, we see
\begin{align*}
{A}_{\nnn}^{(i,1)}
=&\frac{\an[\sum_{k=1}^3 n_k]{\frac{5}{2}-1} \prod_{k=1}^3 \an[n_k]{\frac{1}{2}}}
{\an[\sum_{k=1}^3 n_k]{3-1} \prod_{k=1}^3 \an[n_k]{1}},
\end{align*}
combined with
\[
\frac{\an[n_1+n_2+n_3]{\frac{3}{2}}}
{\an[n_1+n_2+n_3]{2}}
=\frac{\an[n_1+n_2+n_3]{\frac{1}{2}}\left(\frac{1}{2}\right)^{-1}\left(\frac{1}{2}+n_1+n_2+n_3\right)}{\an[n_1+n_2+n_3]{1}(1+n_1+n_2+n_3)},
\]
leads to \eqref{eq:ALinterchanging}.

\eqref{eq:ALinterchanging2} follows from
\begin{align}
{A}_{\nnn[+1][][]}^{(i,2)}
&=\frac{(\frac{1}{2}+n_1+n_2+n_3)(\frac{1}{2}+n_1)}{(1+n_1+n_2+n_3)(1+n_1)}{A}_{\nnn}^{(i,2)}\label{eq:ALrecrel}\\
&=\frac{(\frac{1}{2}+n_1+n_2+n_3)(\frac{1}{2}+n_1)}{(1+n_1+n_2+n_3)(1+n_1)}\cdot\frac{1+n_1+n_2+n_3}{1+2(n_1+n_2+n_3)}{A}_{\nnn}^{(i,1)}.\notag
\end{align}
\end{proof}
\end{Lem}

\subsection{Superspeciality implies nonsingularity}
Let $K$ be a field of characteristic $p>2$.
Let $C$ be the curve
\[
y^2 = f(x):=x(x-1)(x-\lambda_1)(x-\lambda_2)(x-\lambda_3)
\]
for $\lambda_1,\lambda_2,\lambda_3 \in K$.
Let $c_m$ be the $x^m$-coefficient of $f(x)^e$ with $e:=(p-1)/2$.
\begin{Prop}\label{prop:SspImpliesNonsing}
If $C$ is superspecial, more precisely its Cartier-Manin matrix
\[
\begin{pmatrix}
c_{p-1} & c_{p-2}\\
c_{2p-1} & c_{2p-2}
\end{pmatrix}
\]
is zero, then
we have $\lambda_i \ne 0, 1$ and $\lambda_i \ne \lambda_j$ if $i\ne j$
for $i,j\in\{1,2,3\}$.
\end{Prop}
\begin{proof}
According to \cite[Lemma 2.4]{KHS},
by any linear coordinate change, say $x \mapsto ux+v$ for constants $u(\ne 0)$ and $v$, the $\sigma$-conjugacy class of the Cartier-Manin matrix (especially whether the Cartier-Manin matrix is zero or not) is unchanged.
Hence, without loss of generality it suffices to show that if $\lambda_3 = 0$, then
the Cartier-Manin matrix is not zero. 
Assume $\lambda_3 = 0$, i.e., 
consider the case $f(x)=x^2(x-1)(x-\lambda_1)(x-\lambda_2)$.
Note that $c_{p-2}=0$ always holds.
It suffices to show that
$c_{p-1}=c_{2p-2}=0$ implies $c_{2p-1} \ne 0$.
Assume $c_{p-1}=c_{2p-2}=0$.
Since $c_{p-1}=(-\lambda_1 \lambda_2)^e$, we may assume $\lambda_2=0$
without loss of generality, whence $f(x)=x^3(x-1)(x-\lambda_1)$.
Put $g(x) = x(x-1)(x-\lambda_1)$. We have
\[
f(x)^e = x^{p-1} g(x)^e
\]
Let $\delta_i$ be the $x^i$-coefficient of $g(x)^e$.
Then $\delta_{p-1}$ is zero, since it is equal to $c_{2p-2}$.
According to \cite[Chap.\ 5, Theorem 4.1]{Silverman} with its proof,
$\lambda_1$ is a solution of $H_p(t)=0$ and $\lambda_1 \ne 0,1$. 
We claim $\delta_p \ne 0$.
This is obtained by the relation $\delta_{p-2} = \lambda_1 \cdot \delta_p$ (that follows from a straightforward computation, the authors learned this relation from Ryo Ohashi) and $\delta_{p-2}\ne 0$ that follows from \cite[Proposition 4.3]{KHS}.
Since $\delta_p$ is the $x^{2p-1}$-coefficient $c_{2p-1}$ of $f(x)^e$, we have $c_{2p-1} \ne 0$.
\end{proof}

\section{The main results}
In this section, we collect the main theorems of this paper.

Let $j\in\{1,2\}$.
We consider 
the partial differential operators $\diff{\ell}$ as in \eqref{eq:Dl}
associated to $a = 5/2 - j$, $(b_1,b_2,b_3)=(1/2,1/2,1/2)$
and $c = 3-j$ for $\ell=1,2,3$.
We denote it by $\diffchar{\ell}{j}$,
because it depends on $j$.

\thmA*


As an application of this theorem and the contiguity relations obtained in Subsection \ref{subsec:Contiguity}, we have the following theorem,
that can be regard as an analogous result in the genus-two case of Igusa's result \cite{Igusa} in the genus-one case.
We call this theorem {\it the multiplicity-one theorem for the superspeciality of the curve $y^2=x(x-1)(x-z_1)(x-z_2)(x-z_3)$}.

\thmB*


\section{Proof of Theorem \ref{thm:PDEinPositiveChar}}
In this section, we give a proof of Theorem \ref{thm:PDEinPositiveChar}.
We use the notation in Definition \ref{def:Truncation}.
\begin{proof}[Proof of Theorem \ref{thm:PDEinPositiveChar}]
To show the first equation,
it suffices to show that $\diffchar{\ell}{j} \truncatedAL = 0$,
since
$c_{ip-j} = (-1)^\frac{p-1}{2}\cdot\dfrac{j}{i} \truncatedAL$.
From the symmetry for $\ell,$ we prove only the case of $\ell=1$.

Let $\Tilde{A}_{\nnn}^{(i,j)}$ be
the $\zzz$-coefficients of $\truncatedAL$. Then
\[
\Tilde{A}_{\nnn}^{(i,j)} = A_{\nnn}^{(i,j)}=\dfrac{ \an[\sum_{k=1}^3 n_k]{a'} \prod_{k=1}^3 \an[n_k]{b'} }
{ \an[\sum_{k=1}^3 n_k]{c'} \prod_{k=1}^3 \an[n_k]{1} }
\]
for $(\nnn)\in\supp,$ and $\Tilde{A}_{\nnn}^{(i,j)}=0$ otherwise.
By Remark \ref{rem:StandardRelation}, we have
\[
\diffchar{1}{j} \truncatedAL
=\sum_{\nnn\ge0} C_{\nnn}^{(i,j)} \zzz,
\]
where $C_{\nnn}^{(i,j)}$ is
\[
(1+n_1)(c'+n_1+n_2+n_3) \Tilde{A}_{\nnn[+1][][]}^{(i,j)}
-(b'+n_1)(a'+n_1+n_2+n_3) \Tilde{A}_{\nnn}^{(i,j)}
\]
for $(\nnn)\in\NNI^3.$ We check that $C_{\nnn}^{(i,j)}$ vanishes for every $(\nnn)\in\NNI^3$, dividing into five cases.

\underline{Case 1.} We consider the case when $\nnn$ satisfies
\[
n_1\le\frac{p-1}{2}-1,\,n_k\le\frac{p-1}{2}\,(k=2,3),\,d'-\frac{p-1}{2}\le n_1+n_2+n_3\le d'-1.
\]
Since
\[
n_1+1\le\frac{p-1}{2},\,d'-\frac{p-1}{2}+1\le (n_1+1)+n_2+n_3\le d',
\]
$(\nnn),(\nnn[+1][][])\in\supp,$ and by the standard recurrence relation, $C_{\nnn}^{(i,j)}=0.$

\underline{Case 2.} We consider the case when $n_1 = (p-1)/{2}$.
Since $n_1+1>(p-1)/2$ deduces $\Tilde{A}_{\nnn[+1][][]}^{(i,j)}=0,$ and $b'+n_1=p/2=0,$ it holds $C_{\nnn}^{(i,j)}=0.$

\underline{Case 3.}  We consider the case when $n_1+n_2+n_3 = d'-(p-1)/{2}-1$. First, note $n_1+n_2+n_3 < d'-(p-1)/2$ deduces $\Tilde{A}_{\nnn}^{(i,j)}=0.$ 
\textcolor{black}{Note $a' + d' = 0$ in ${\mathbb F}_p$ in general}.
Since $c'+n_1+n_2+n_3 = (a'+1/2)+ ( d'-(p-1)/2-1 ) = a'+d' = 0,$ it holds $C_{\nnn}^{(i,j)}=0.$

\underline{Case 4.}  We consider the case when $n_1+n_2+n_3 = d'$. Since $d' < (1+n_1)+n_2+n_3$ deduces  $\Tilde{A}_{\nnn[+1][][]}^{(i,j)}=0,$ and $a'+n_1+n_2+n_3 = a'+d' = 0,$ it holds $C_{\nnn}^{(i,j)}=0.$

\underline{Case 5.} We consider the case when $\nnn$ satisfies $(p-1)/2 < n_1$ or $(p-1)/2 < n_2$ or $(p-1)/2 < n_3$ or $ n_1+n_2+n_3\le d'-(p-1)/2-2$ or $ d'+1 \le n_1+n_2+n_3.$ In this case, neither $(\nnn)$ nor $(\nnn[+1][][])$ belongs to $\supp.$ Hence, $\Tilde{A}_{\nnn}^{(i,j)}=\Tilde{A}_{\nnn[+1][][]}^{(i,j)}=0,$ and it holds $C_{\nnn}^{(i,j)}=0.$

To show the second equation,
we only need to show $\diffsimp{\ell}{m}\truncatedAL=0$, since 
$c_{ip-j} = (-1)^\frac{p-1}{2}\cdot\dfrac{j}{i} \truncatedAL$.  From the symmetry for $(\ell, m)$, it suffices to consider the case of $(\ell,m)=(1,2)$. 
The $\zzz$-coefficient of
\begin{eqnarray*}
\diffsimp{1}{2}\truncatedAL
&=&\sum_{\nnn\ge0}
    \Bigl(
        n_1 n_2(\zzz[][-1][]-\zzz[-1][][])\\
&&\hspace{4em} +\frac{1}{2}n_2\zzz[][-1][]-\frac{1}{2}n_1\zzz[-1][][]
    \Bigr)
    \Tilde{A}_{\nnn}^{(i,j)}
\end{eqnarray*}
is
\[
\left(n_1+\frac{1}{2}\right) \left(n_2+1\right)\Tilde{A}_{\nnn[][+1][]}^{(i,j)}-\left(n_1+1\right)\left(n_2+\frac{1}{2}\right)\Tilde{A}_{\nnn[+1][][]}^{(i,j)}.
\]
This vanishes by the standard recurrence relations in Lemma \ref{lem:recurrence_relations} \textcolor{black}{(1) for $n_2$ and $n_1$}.
\end{proof}

\section{Proof of the multiplicity-one theorem
}
In this section, we prove Theorem \ref{thm:MultiplicityOne}.

\subsection{The Jacobian criterion}
Let $k$ be an algebraically closed field.
Let $R=k[x_1,\ldots,x_n]$. Let $I$ be an ideal of $R$ such that $R/I$ is an Artin ring
(i.e., $I$ is zero dimensional).
\begin{Def}
We say $I$ is of {\it multiplicity-one} if  $\dim_k (R/I)_\mathfrak{P} = 1$ for any $\mathfrak{P}$ in $\mathrm{Spec}(R/I)$.
\end{Def}
\begin{Lem}
$I=(f_1,\ldots,f_m)$ is muplicplicity one if and only if for any $\mathfrak{P}$ in $\mathrm{Spec}(R/I)$
the Jacobian matrix $\left(\dfrac{\partial f_j}{\partial x_i}\right)$ at $\mathfrak{P}$ is of rank $n$.
\end{Lem}
\begin{proof}
Let $I = \bigcap_i \mathfrak{q}_i$
be the irredundant primary decomposition of $I$. Set $\mathfrak{p}_i = \sqrt{\mathfrak{q}_i}$.
Since $A/I$ is an Artin ring, $\mathfrak{p}_i$ is a maximal ideal.
Hence we have $\mathfrak{p}_i + \mathfrak{p}_j = R$ for $i \ne j$ and therefore
$\mathfrak{q}_i + \mathfrak{q}_j = R$. By the Chinese remainder theorem, we get
an isomorphism $\Psi: R/I \simeq \bigoplus_i R/\mathfrak{q}_i$.

We claim that $(R/I)_\mathfrak{P} \simeq R/\mathfrak{q}_i$ for $i$ with $\sqrt{\mathfrak{q}_i} = \mathfrak{p}$,
where $\mathfrak p$ is the inverse image of $\mathfrak{P}$ by $R \to R/I$.
Indeed $\Psi$ sends $\mathfrak{P}$ (resp. $S:=(R/I) \setminus \mathfrak{P}$) to the product of $\mathfrak p/\mathfrak{q}_i$ (resp.\ $(R/\mathfrak{q}_i)^\times$) and $R/\mathfrak{q}_j$ for $j\ne i$.
Since the idempotent on $\bigoplus_i R/\mathfrak{q}_i$ which is the identity on $R/\mathfrak{q}_i$ and zero on $R/\mathfrak{q}_j$ for $j\ne i$
belongs to $S$, the localization $(R/I)_\mathfrak{P}$ is equal to the localization
$(R/\mathfrak{q}_i)_{\mathfrak p/\mathfrak{q}_i} = R/\mathfrak{q}_i$.

Clearly $(R/I)_\mathfrak{P} \simeq R/\mathfrak{q}_i$ is a one-dimensional $k$-vector space if and only if $\mathfrak{q}_i = \mathfrak{p}$.
This is equivalent to that $\mathrm{Spec}(R/I)$ is nonsingular. The Jacobian criterion for the nonsingularity implies the desired assertion.
\end {proof}

\subsection{Contiguity relations}\label{subsec:Contiguity}
Inspired by the contiguity relations in \cite[Theorem 3.8.1]{Matsumoto}
of Lauricella hypergeometric series,
we have some relations on $c_{ip-j}$
with some partial differential operators.
See Remark \ref{rem:OriginalContiguityRelations} for the relations directly predicted from 
\cite[Theorem 3.8.1]{Matsumoto} and Theorem \ref{thm:PDEinPositiveChar},
which we do not use for the proof of Theorem \ref{thm:MultiplicityOne}.
The relations necessary for the proof of Theorem \ref{thm:MultiplicityOne}
are those in the next theorem, which look more mysterious.
The \textcolor{black}{authors} found the relations with the observation that the left-hand sides
vanish at all the superspecial points,
i.e., it must belong to the ideal generated by $c_{ip-j}$.
The second relation is reminiscent of \cite[Lemma 4.2]{KHS} in the genus-one case.

\begin{Thm}\label{thm:ContiguityRelation}
For $i=1,2$, the following equalities hold:
\begin{enumerate}
\item[(1)]
$\displaystyle
\left(\sum_{k=1}^3 (z_k^2-z_k)\partial_k\right) c_{ip-1} =-\dfrac{1}{2}(z_1+z_2+z_3-2)c_{ip-1}-\dfrac{1}{2}c_{ip-2},
$
\item[(2)]
$\displaystyle
\left(\sum_{k=1}^3 (1-z_k)\partial_k\right) c_{ip-2} = \dfrac{1}{2}(c_{ip-1}+c_{ip-2}).
$
\end{enumerate}
\end{Thm}
\begin{proof}

We use the standard unit vectors:
\[
\bm{e}_1 = (1,0,0),\,\bm{e}_2 = (0,1,0),\,\bm{e}_3 = (0,0,1)\textcolor{black}{.}
\]
\textcolor{black}{For} an element \textcolor{black}{$\bm{n}=(n_1,n_2,n_3)$} of $\NNI^3$
\textcolor{black}{let $\left| \bm{n} \right| = n_1+n_2+n_3$ denote the total degree of the monomial $\bm{z}^{\bm{n}}= z_1^{n_1} z_2^{n_2} z_3^{n_3}$,} and
\textcolor{black}{for} a subset $S$ of $\NNI^3$ we denote 
the set $\{\bm{\textcolor{black}{n}}+s | s\in S\}$ by $\bm{\textcolor{black}{n}}+S$,
which will be used for $S=\supp$, see \eqref{eq:supp} for the definition of $\supp$.

(1)
\textcolor{black}{
Let $\varphi(z_1,z_2,z_3) \in {\mathbb{F}_p}[z_1,z_2,z_3]$ be
\[
(-1)^\frac{p-1}{2}i\left(\sum_{k=1}^3 (z_k^2-z_k)\partial_k c_{ip-1}+\dfrac{1}{2}(z_1+z_2+z_3-2)c_{ip-1}+\dfrac{1}{2}c_{ip-2}\right).
\]
We show that $\varphi$ is zero.}
\textcolor{black}{After some simple computations,}
\textcolor{black}{$\varphi$ is equal to}
\begin{equation*}
\sum_{k=1}^3\sum_{\textcolor{black}{\bm{n}\in}\bm{e}_k+\supp[i][1]}
\left(n_k-\frac{1}{2}\right)\Tilde{A}_{\textcolor{black}{\bm{n}}-\bm{e}_k}^{(i,1)}\textcolor{black}{\bm{z}^{\bm{n}}}
-\sum_{\textcolor{black}{\bm{n}\in}\supp[i][1]}(1+\textcolor{black}{\left| \bm{n} \right|})\Tilde{A}_{\textcolor{black}{\bm{n}}}^{(i,1)}\textcolor{black}{\bm{z}^{\bm{n}}} + \sum_{\textcolor{black}{\bm{n}\in}\supp[i][2]}\Tilde{A}_{\textcolor{black}{\bm{n}}}^{(i,2)}\textcolor{black}{\bm{z}^{\bm{n}}}.
\end{equation*}
In order to distinguish the terms involved in the translated regions from those not, we \textcolor{black}{denote}, for $i=1,2$ and for $k=1,2,3,$
\begin{equation*}
\left(\bm{e}_k+\supp[i][1]\right)\bigcap\supp[i][1],
\left(\bm{e}_k+\supp[i][1]\right) \setminus \supp[i][1],
\supp[i][1] \setminus \left(\bm{e}_k+\supp[i][1]\right)
\end{equation*}
\textcolor{black}{by $\mathsf{R}^{(i,1)}_k, \mathsf{S}^{(i,1)}_k, \mathsf{T}^{(i,1)}_k$ respectively.}
Using the disjoint union $\bm{e}_k+\supp[i][1] = \mathsf{R}^{(i,1)}_k \sqcup \mathsf{S}^{(i,1)}_k,$ \textcolor{black}{we see that $\varphi$ is}
\begin{align}
&\sum_{k=1}^3\sum_{\textcolor{black}{\bm{n}\in}\mathsf{R}^{(i,1)}_k}
    \left(n_k-\frac{1}{2}\right)\Tilde{A}_{\textcolor{black}{\bm{n}}-\bm{e}_k}^{(i,1)}\textcolor{black}{\bm{z}^{\bm{n}}} + \sum_{k=1}^3\sum_{\textcolor{black}{\bm{n}\in}\mathsf{S}^{(i,1)}_k}
    \left(n_k-\frac{1}{2}\right)\Tilde{A}_{\textcolor{black}{\bm{n}}-\bm{e}_k}^{(i,1)}\textcolor{black}{\bm{z}^{\bm{n}}}\notag\\
&-\sum_{\textcolor{black}{\bm{n}\in}\supp[i][1]}(1+\textcolor{black}{\left| \bm{n} \right|}
)\Tilde{A}_{\textcolor{black}{\bm{n}}}^{(i,1)}\textcolor{black}{\bm{z}^{\bm{n}}}+\sum_{\textcolor{black}{\bm{n}\in}\supp[i][2]}\Tilde{A}_{\textcolor{black}{\bm{n}}}^{(i,2)}\textcolor{black}{\bm{z}^{\bm{n}}}.\label{eq:intermediate_f}
\end{align}
\textcolor{black}{We note that for} any $\textcolor{black}{\bm{n}=}(\nnn) \in \supp[i][1]$ and for any $k=1,2,3,$ it follows, by using the standard recurrence relation, that
\[
\left(n_k-\frac{1}{2}\right)\Tilde{A}_{\textcolor{black}{\bm{n}}-\bm{e}_k}^{(i,1)}\textcolor{black}{\bm{z}^{\bm{n}}}\\
=\frac{(1+\textcolor{black}{\left| \bm{n} \right|}
)n_k}{\frac{1}{2}+\textcolor{black}{\left| \bm{n} \right|}
}\Tilde{A}_{\textcolor{black}{\bm{n}}}^{(i,1)}\textcolor{black}{\bm{z}^{\bm{n}}},
\]
which can be considered to be true even when $n_k = 0.$
\textcolor{black}{After some computations using $\mathsf{R}^{(i,1)}_k = \supp[i][1] \setminus \mathsf{T}^{(i,1)}_k\,(\text{note } \supp[i][1] = \mathsf{R}^{(i,1)}_k \sqcup \mathsf{T}^{(i,1)}_k)\,(k=1,2,3)$, $\varphi=$ \eqref{eq:intermediate_f} is}
\begin{align}
&\sum_{k=1}^3\sum_{\textcolor{black}{\bm{n}\in}\mathsf{S}^{(i,1)}_k}
    \left(n_k-\frac{1}{2}\right)\Tilde{A}_{\textcolor{black}{\bm{n}}-\bm{e}_k}^{(i,1)}\textcolor{black}{\bm{z}^{\bm{n}}} -\sum_{k=1}^3\sum_{\textcolor{black}{\bm{n}\in}\mathsf{T}^{(i,1)}_k}
    \frac{(1+\textcolor{black}{\left| \bm{n} \right|}
)n_k}{\frac{1}{2}+\textcolor{black}{\left| \bm{n} \right|}
}\Tilde{A}_{\textcolor{black}{\bm{n}}}^{(i,1)}\textcolor{black}{\bm{z}^{\bm{n}}}\notag\\
&-\frac{1}{2}\sum_{\textcolor{black}{\bm{n}\in}\supp[i][1]}\frac{1+\textcolor{black}{\left| \bm{n} \right|}
}{\frac{1}{2}+\textcolor{black}{\left| \bm{n} \right|}
}\Tilde{A}_{\textcolor{black}{\bm{n}}}^{(i,1)}\textcolor{black}{\bm{z}^{\bm{n}}} + \sum_{\textcolor{black}{\bm{n}\in}\supp[i][2]}\Tilde{A}_{\textcolor{black}{\bm{n}}}^{(i,2)}\textcolor{black}{\bm{z}^{\bm{n}}}.\label{eq:contiguity1}
\end{align}
The proof of the vanishing of this sum splits into two cases (Case $i=1$ and Case $i=2$).
Before we get into the details, we explain briefly how
the terms cancel out:
Every term of the sum involving $\mathsf{T}^{(i,1)}_k$ vanishes itself for each $i$. In the case of $i=1$, every term of the sum involving $\mathsf{S}^{(1,1)}_k$ also vanishes itself and the other two sums cancel out with the help of the standard recurrence relations. In the case of $i=2$, the sum involving $\mathsf{S}^{(2,1)}_k$ and a portion of the sum involving $\supp[2][2]$ are combined to vanish and the other remaining sums vanish by the standard \textcolor{black}{recurrence} relations. Here is the detailed proof:

\underline{Case $i=1.$}
We note
\begin{align*}
\supp[1][1] =&\setsymbolin{\textcolor{black}{\bm{n}=}(n_1,n_2,n_3)}{\NNI^3}{\begin{array}{l}
    n_k \le \textcolor{black}{(p-1)/2} \ (k=1,2,3),\\
    p-1 \le \textcolor{black}{\left| \bm{n} \right|}
 \le \textcolor{black}{(3p-3)/2}
\end{array}},
\end{align*}
\textcolor{black}{and $\supp[1][2]$ is equal to}
\begin{equation*}
\setsymbolin{\textcolor{black}{\bm{n}}}{\NNI^3}{\begin{array}{l}
    n_k \le \textcolor{black}{(p-1)/2} \ (\textcolor{black}{\forall k}),\\
    p \le \textcolor{black}{\left| \bm{n} \right|} \le \textcolor{black}{(3p-1)/2}
\end{array}}
=\setsymbol{\textcolor{black}{\bm{n}}}{\begin{array}{l}
    n_k \le \textcolor{black}{(p-1)/2} \ (\textcolor{black}{\forall k}),\\
    p \le \textcolor{black}{\left| \bm{n} \right|} \le \textcolor{black}{(3p-3)/2}
\end{array}}.
\end{equation*}
\textcolor{black}{Moreover,} for $k=1,2,3$,
\begin{align*}
\mathsf{S}^{(1,1)}_k&=\left(\bm{e}_k+\supp[1][1]\right) \setminus \supp[1][1]
= \bm{e}_k + \setsymbol{\textcolor{black}{\bm{n}}\in\supp[1][1]}{n_k=\textcolor{black}{(p-1)/2}},\\
\mathsf{T}^{(1,1)}_k&=\supp[1][1] \setminus \left(\bm{e}_k+\supp[1][1]\right)
= \setsymbol{\textcolor{black}{\bm{n}}\in\supp[1][1]}{\textcolor{black}{\left| \bm{n} \right|}=p-1}.
\end{align*}
Note $\mathsf{T}^{(1,1)}_k$ is independent of $k=1,2,3.$
For each $(n'_1,n'_2,n'_3)\in\mathsf{S}^{(1,1)}_k,$
\textcolor{black}{$n'_k - 1/2 = (p+1)/2 - 1/2 \equiv 0 \pmod{p}$,}
\textcolor{black}{which implies}
\[
\sum_{\textcolor{black}{\bm{n}\in}\mathsf{S}^{(1,1)}_k}
    \left(n_k-\frac{1}{2}\right)\Tilde{A}_{\textcolor{black}{\bm{n}}-\bm{e}_k}^{(1,1)}\textcolor{black}{\bm{z}^{\bm{n}}} = 0;
\]
 \textcolor{black}{furthermore}, for each $\textcolor{black}{\bm{n'}=}(n'_1,n'_2,n'_3)\in\mathsf{T}^{(1,1)}_k,$ it holds true that
\begin{align*}
&\frac{(1+\textcolor{black}{\left| \bm{n'} \right|})n'_k}{\frac{1}{2}+\textcolor{black}{\left| \bm{n'} \right|}}\Tilde{A}_{\textcolor{black}{\bm{n'}}}^{(1,1)}\textcolor{black}{\bm{z}^{\bm{n'}}}
=\frac{\left(1+(p-1)\right)n'_k}{\frac{1}{2}+(p-1)}\Tilde{A}_{\textcolor{black}{\bm{n'}}}^{(1,1)}\textcolor{black}{\bm{z}^{\bm{n'}}}=0.
\end{align*}
Hence, \textcolor{black}{\eqref{eq:contiguity1} is equal to}
\begin{align*}
&-\frac{1}{2}\sum_{\textcolor{black}{\bm{n}\in}\supp[1][1]}\frac{1+\textcolor{black}{\left| \bm{n} \right|}}{\frac{1}{2}+\textcolor{black}{\left| \bm{n} \right|}}\Tilde{A}_{\textcolor{black}{\bm{n}}}^{(1,1)}\textcolor{black}{\bm{z}^{\bm{n}}}
+\sum_{\textcolor{black}{\bm{n}\in}\supp[1][2]}\Tilde{A}_{\textcolor{black}{\bm{n}}}^{(1,2)}\textcolor{black}{\bm{z}^{\bm{n}}}\\
=&-\frac{1}{2}\sum_{\textcolor{black}{\bm{n}\in}\supp[1][2]}\frac{1+\textcolor{black}{\left| \bm{n} \right|}}{\frac{1}{2}+\textcolor{black}{\left| \bm{n} \right|}}\Tilde{A}_{\textcolor{black}{\bm{n}}}^{(1,1)}\textcolor{black}{\bm{z}^{\bm{n}}}
+\sum_{\textcolor{black}{\bm{n}\in}\supp[1][2]}\Tilde{A}_{\textcolor{black}{\bm{n}}}^{(1,2)}\textcolor{black}{\bm{z}^{\bm{n}}}.
\end{align*}
Then, by \eqref{eq:ALinterchanging}, the last line above annihilates.

\underline{Case $i=2.$}
We note
\begin{align*}
\supp[2][1]&=\setsymbolin{\textcolor{black}{\bm{n}}}{\NNI^3}{-1 \le \textcolor{black}{\left| \bm{n} \right|} \le \textcolor{black}{(p-3)/2}
}
=\setsymbol{\textcolor{black}{\textcolor{black}{\bm{n}}}}{\textcolor{black}{\left| \bm{n} \right|} \le \textcolor{black}{(p-3)/2}},\\
\supp[2][2]&=\setsymbolin{\textcolor{black}{\bm{n}}}{\NNI^3}{\textcolor{black}{\left| \bm{n} \right|} \le \textcolor{black}{(p-1)/2}},
\end{align*}
and, for $k=1,2,3$,
\begin{align*}
\mathsf{S}^{(2,1)}_k&=\left(\bm{e}_k+\supp[2][1]\right) \setminus \supp[2][1]
= \bm{e}_k + \setsymbol{\textcolor{black}{\bm{n}}\in\supp[2][1]}{\textcolor{black}{\left| \bm{n} \right|}=\textcolor{black}{(p-3)/2}},\\
\mathsf{T}^{(2,1)}_k&=\supp[2][1] \setminus \left(\bm{e}_k+\supp[2][1]\right)
= \setsymbol{\textcolor{black}{\bm{n}}\in\supp[2][1]}{n_k=0}.
\end{align*}
The fact that for each $(n'_1,n'_2,n'_3)\in\mathsf{T}^{(2,1)}_k,$ $n'_k=0\quad(k=1,2,3)$ implies
\[
\sum_{k=1}^3\sum_{\textcolor{black}{\bm{n}\in}\mathsf{T}^{(2,1)}_k}
    \frac{(1+\textcolor{black}{\left| \bm{n} \right|})n_k}{\frac{1}{2}+\textcolor{black}{\left| \bm{n} \right|}}\Tilde{A}_{\textcolor{black}{\bm{n}}}^{(2,1)}\textcolor{black}{\bm{z}^{\bm{n}}} = 0.
\]
We see that
\begin{equation*}
\supp[2][2] \setminus \supp[2][1]
= \setsymbolin{\textcolor{black}{\bm{n}}}{\NNI^3}{\textcolor{black}{\left| \bm{n} \right|} = \textcolor{black}{(p-1)/2}}
=\bigcup_{k=1}^3\mathsf{S}^{(2,1)}_k.
\end{equation*}
Here, we claim that the following \textcolor{black}{part of \eqref{eq:contiguity1} is zero:}
\begin{equation}
\sum_{k=1}^3\sum_{\textcolor{black}{\bm{n}\in}\mathsf{S}^{(2,1)}_k}\left(n_k-\frac{1}{2}\right)\Tilde{A}_{\textcolor{black}{\bm{n}}-\bm{e}_k}^{(2,1)}\textcolor{black}{\bm{z}^{\bm{n}}} +\sum_{\textcolor{black}{\bm{n}\in}\supp[2][2] \setminus \supp[2][1]}\Tilde{A}_{\textcolor{black}{\bm{n}}}^{(2,2)}\textcolor{black}{\bm{z}^{\bm{n}}}
=0.\label{claim:surplusvanishing}
\end{equation}
We prove this claim below. We take an arbitrary exponent $\textcolor{black}{\bm{n'}=}(n'_1,n'_2,n'_3) \in \supp[2][2] \setminus \supp[2][1]$ of monomials that appear in \eqref{claim:surplusvanishing}, and we check the coefficient of the monomial $\textcolor{black}{\bm{z}^{\bm{n'}}}$ is zero. Note since $\textcolor{black}{\left| \bm{n'} \right|}=(p-1)/2\ge1$, one of $n'_k$ is \textcolor{black}{nonzero}. When exactly one of $n'_k$ is \textcolor{black}{nonzero}, from symmetry we can assume that $n'_1=(p-1)/2;n'_2=n'_3=0.$ In this case, we see that
\[
\textcolor{black}{\bm{n'}} \in \mathsf{S}^{(2,1)}_1,\,\textcolor{black}{\bm{n'}} \notin\mathsf{S}^{(2,1)}_k\ (k=2,3),
\]
and therefore, using \eqref{eq:ALinterchanging2}, the $\textcolor{black}{\bm{z}^{\bm{n'}}}$-coefficient is
\begin{equation*}
\left(n'_1-\frac{1}{2}\right)\Tilde{A}_{\textcolor{black}{\bm{n'}}-\bm{e_1}}^{(2,1)}+\Tilde{A}_{\textcolor{black}{\bm{n'}}}^{(2,2)}
=-\Tilde{A}_{\frac{p-3}{2},0,0}^{(2,1)}+\frac{\frac{1}{2}+\frac{p-3}{2}}{2\left(1+\frac{p-3}{2}\right)}\Tilde{A}_{\frac{p-3}{2},0,0}^{(2,1)}
=0.
\end{equation*}
When exactly two of $n'_k$ are \textcolor{black}{nonzero}, from symmetry we can assume that $n'_1+n'_2=(p-1)/2$ and $n'_1,n'_2\ge1,n'_3=0.$ In this case, we see that
\[
\textcolor{black}{\bm{n'}} \in \mathsf{S}^{(2,1)}_k\ (k=1,2)\textcolor{black}{,}\,\textcolor{black}{\bm{n'}} \notin\mathsf{S}^{(2,1)}_3,
\]
and therefore, using \eqref{eq:ALinterchanging2}, the $\textcolor{black}{\bm{z}^{\bm{n'}}}$-coefficient is
\[
\left(n'_1-\frac{1}{2}\right)\Tilde{A}_{\textcolor{black}{\bm{n'}-\bm{e_1}}}^{(2,1)}+\left(n'_2-\frac{1}{2}\right)\Tilde{A}_{\textcolor{black}{\bm{n'}-\bm{e_2}}}^{(2,1)}+\Tilde{A}_{\textcolor{black}{\bm{n'}}}^{(2,2)}.
\]
\textcolor{black}{We have \eqref{claim:surplusvanishing} by computing the first two terms of the above coefficient using \eqref{eq:ALinterchanging2} as follows:}
\begin{align*}
&\left(n'_1-\frac{1}{2}\right)\Tilde{A}_{\textcolor{black}{\bm{n'}-\bm{e_1}}}^{(2,1)}+\left(n'_2-\frac{1}{2}\right)\Tilde{A}_{\textcolor{black}{\bm{n'}-\bm{e_2}}}^{(2,1)}\\
&=\left(n'_1-\frac{1}{2}\right)\cdot\frac{2n'_1}{n'_1-\frac{1}{2}}\Tilde{A}_{\textcolor{black}{\bm{n'}}}^{(2,2)}+\left(n'_2-\frac{1}{2}\right)\cdot\frac{2n'_2}{n'_2-\frac{1}{2}}\Tilde{A}_{\textcolor{black}{\bm{n'}}}^{(2,2)}\\
&=2(n'_1+n'_2)\Tilde{A}_{\textcolor{black}{\bm{n'}}}^{(2,2)}
=2\cdot\frac{p-1}{2}\Tilde{A}_{\textcolor{black}{\bm{n'}}}^{(2,2)}
=-\Tilde{A}_{\textcolor{black}{\bm{n'}}}^{(2,2)}.
\end{align*}
When none of $n'_k$ are \textcolor{black}{nonzero}, $\textcolor{black}{\left| \bm{n'} \right|}=(p-1)/2$ and $n'_k\ge1\,(k=1,2,3).$ In this case, the reasoning is the same as in the last one.

Now \textcolor{black}{the claim \eqref{claim:surplusvanishing}} is verified, \textcolor{black}{we see that using the disjoint union $\supp[2][2] = \supp[2][1] \sqcup \left(\supp[2][2] \setminus \supp[2][1]\right),$ \eqref{eq:contiguity1} is equal to}
\begin{align*}
&-\frac{1}{2}\sum_{\textcolor{black}{\bm{n}\in}\supp[2][1]}\frac{1+\textcolor{black}{\left| \bm{n} \right|}}{\frac{1}{2}+\textcolor{black}{\left| \bm{n} \right|}}\Tilde{A}_{\textcolor{black}{\bm{n}}}^{(2,1)}\textcolor{black}{\bm{z}^{\bm{n}}}
+\sum_{\textcolor{black}{\bm{n}\in}\supp[2][2]\setminus\left(\supp[2][2] \setminus \supp[2][1]\right)}\Tilde{A}_{\textcolor{black}{\bm{n}}}^{(2,2)}\textcolor{black}{\bm{z}^{\bm{n}}}\\
=&-\frac{1}{2}\sum_{\textcolor{black}{\bm{n}\in}\supp[2][1]}\frac{1+\textcolor{black}{\left| \bm{n} \right|}}{\frac{1}{2}+\textcolor{black}{\left| \bm{n} \right|}}\Tilde{A}_{\textcolor{black}{\bm{n}}}^{(2,1)}\textcolor{black}{\bm{z}^{\bm{n}}}
+\sum_{\textcolor{black}{\bm{n}\in}\supp[2][1]}\Tilde{A}_{\textcolor{black}{\bm{n}}}^{(2,2)}\textcolor{black}{\bm{z}^{\bm{n}}}.
\end{align*}
Then, by \eqref{eq:ALinterchanging}, the last line above annihilates.

(2)
\textcolor{black}{
Let $\psi(z_1,z_2,z_3) \in {\mathbb{F}_p}[z_1,z_2,z_3]$ be
\[
(-1)^\frac{p-1}{2}\frac{i}{2}\cdot\left(\sum_{k=1}^3 (1-z_k)\partial_k c_{ip-2}-\dfrac{1}{2}(c_{ip-1}+c_{ip-2})\right).
\]
We show that $\psi$ is zero.}
For $i=1,2,$ and $\,\,k=1,2,3,$ we define
\begin{equation*}
\mathsf{U}^{(i,2)}_k:=\left(-\bm{e}_k+\supp[i][2]\right)\bigcap\supp[i][2],
\mathsf{V}^{(i,2)}_k:=\left(-\bm{e}_k+\supp[i][2]\right) \setminus \supp[i][2].
\end{equation*}
We use a similar strategy to (1), and \textcolor{black}{by simple computations} \textcolor{black}{$\psi$ is}
\begin{align}
&\sum_{k=1}^3\sum_{\textcolor{black}{\bm{n}\in}\mathsf{U}^{(i,2)}_k}\left(n_k+1\right)\Tilde{A}_{\textcolor{black}{\bm{n}}+\bm{e}_k}^{(i,2)}\textcolor{black}{\bm{z}^{\bm{n}}}
+\sum_{k=1}^3\sum_{\textcolor{black}{\bm{n}\in}\mathsf{V}^{(i,2)}_k}\left(n_k+1\right)\Tilde{A}_{\textcolor{black}{\bm{n}}+\bm{e}_k}^{(i,2)}\textcolor{black}{\bm{z}^{\bm{n}}}\notag\\
&-\sum_{\textcolor{black}{\bm{n}\in}\supp[i][2]}(\textcolor{black}{\frac{1}{2} + \left| \bm{n} \right|})\Tilde{A}_{\textcolor{black}{\bm{n}}}^{(i,2)}\textcolor{black}{\bm{z}^{\bm{n}}} \textcolor{black}{-\frac{1}{4}\sum_{\textcolor{black}{\bm{n}\in}\supp[i][1]}\Tilde{A}_{\bm{n}}^{(i,1)}\bm{z}^{\bm{n}}}.\label{eq:contiguity2}
\end{align}
Using the standard recurrence relations,
\begin{equation}
\sum_{k=1}^3\sum_{\textcolor{black}{\bm{n}\in}\mathsf{U}^{(i,2)}_k}\left(n_k+1\right)\Tilde{A}_{\textcolor{black}{\bm{n}}+\bm{e}_k}^{(i,2)}\textcolor{black}{\bm{z}^{\bm{n}}}
=\sum_{k=1}^3\sum_{\textcolor{black}{\bm{n}\in}\mathsf{U}^{(i,2)}_k}
    \frac{\frac{1}{2}+\textcolor{black}{\left| \bm{n} \right|}}{1+\textcolor{black}{\left| \bm{n} \right|}}\cdot\left(\frac{1}{2}+n_k\right)\Tilde{A}_{\textcolor{black}{\bm{n}}}^{(i,2)}\textcolor{black}{\bm{z}^{\bm{n}}}\label{eq:U_part}.
\end{equation}
Here, the proof splits into two cases regarding $i$.

\underline{Case $i=1.$}
We note $\mathsf{V}^{(1,2)}_k=-\bm{e}_k+\setsymbolin{\textcolor{black}{\bm{n}}}{\supp[1][2]}{n_k=0\ \text{or}\ \textcolor{black}{\left| \bm{n} \right|}=p}.$
Setting $\Tilde{\mathsf{V}}^{(1,2)}:=\setsymbolin{\textcolor{black}{\bm{n}}}{\supp[1][2]}{\textcolor{black}{\left| \bm{n} \right|}=p},$ we simplify
\begin{equation}
\sum_{k=1}^3\sum_{\textcolor{black}{\bm{n}\in}\mathsf{V}^{(1,2)}_k}\left(n_k+1\right)\Tilde{A}_{\textcolor{black}{\bm{n}}+\bm{e}_k}^{(1,2)}\textcolor{black}{\bm{z}^{\bm{n}}}
=\sum_{k=1}^3\sum_{\textcolor{black}{\bm{n}\in}-\bm{e}_k+\Tilde{\mathsf{V}}^{(1,2)}}\left(n_k+1\right)\Tilde{A}_{\textcolor{black}{\bm{n}}+\bm{e}_k}^{(1,2)}\textcolor{black}{\bm{z}^{\bm{n}}}.\label{eq:V-tilde2}
\end{equation}
If there existed some $\textcolor{black}{\bm{n'}=(n'_1,n'_2,n'_3)}\in\Tilde{\mathsf{V}}^{(1,2)}$ such that $n\textcolor{black}{'}_k=0$ for some $k\in\{1,2,3\},$ then $\textcolor{black}{\left| \bm{n'} \right|}\le2\cdot(p-1)/2=p-1,$ which contradicts $\textcolor{black}{\left| \bm{n'} \right|}=p.$ From this observation, for $k=1,2,3,$ we obtain
\begin{equation*}
-\bm{e}_k+\Tilde{\mathsf{V}}^{(1,2)}
= \setsymbolin{\textcolor{black}{\bm{n}}}{\supp[1][1]}{\begin{array}{l}
    \textcolor{black}{\left| \bm{n} \right|} = p-1,\\
    n_k \ne (p-1)/2
\end{array}}
=: \Tilde{\mathsf{V}}^{(1,1)}_k
\end{equation*}
Hence, by \eqref{eq:ALinterchanging2}, \textcolor{black}{\eqref{eq:V-tilde2} is equal to}
\begin{equation}
\sum_{k=1}^3\sum_{\textcolor{black}{\bm{n}\in}\Tilde{\mathsf{V}}^{(1,1)}_k}\left(n_k+1\right)\Tilde{A}_{\textcolor{black}{\bm{n}}+\bm{e}_k}^{(1,2)}\textcolor{black}{\bm{z}^{\bm{n}}}
=\sum_{k=1}^3\sum_{\textcolor{black}{\bm{n}\in}\Tilde{\mathsf{V}}^{(1,1)}_k}\frac{1}{2}\left(\frac{1}{2}+n_k\right)\Tilde{A}_{\textcolor{black}{\bm{n}}}^{(1,1)}\textcolor{black}{\bm{z}^{\bm{n}}}.\label{eq:V-tilde}
\end{equation}
Note that we have changed the Lauricella series' coefficients from $j=2$ to $j=1.$ The sum above will cancel out with a portion of $c_{p-1}$ multiplied by a certain value.
Moreover, since the inside of the sum regarding $k$ in the \textcolor{black}{right-hand side of the above formula} won't be affected even if $n_k=(p-1)/2,$ we can replace $\Tilde{\mathsf{V}}^{(1,1)}_k$ for $\Tilde{\mathsf{V}}^{(1,1)}:=\setsymbolin{\textcolor{black}{\bm{n}}}{\supp[1][1]}{\textcolor{black}{\left| \bm{n} \right|} = p-1},$
and \textcolor{black}{\eqref{eq:V-tilde} is}
\begin{equation*}
\sum_{k=1}^3\sum_{\textcolor{black}{\bm{n}\in}\Tilde{\mathsf{V}}^{(1,1)}}\frac{1}{2}\left(\frac{1}{2}+n_k\right)\Tilde{A}_{\textcolor{black}{\bm{n}}}^{(1,1)}\textcolor{black}{\bm{z}^{\bm{n}}}
=\sum_{\textcolor{black}{\bm{n}\in}\Tilde{\mathsf{V}}^{(1,1)}}\frac{1}{2}\left(\frac{3}{2}+\textcolor{black}{\left| \bm{n} \right|}\right)\Tilde{A}_{\textcolor{black}{\bm{n}}}^{(1,1)}\textcolor{black}{\bm{z}^{\bm{n}}}
=\sum_{\textcolor{black}{\bm{n}\in}\Tilde{\mathsf{V}}^{(1,1)}}\frac{1}{4}\Tilde{A}_{\textcolor{black}{\bm{n}}}^{(1,1)}\textcolor{black}{\bm{z}^{\bm{n}}}.
\end{equation*}
Next, we claim that we can replace $\mathsf{U}^{(1,2)}_k$ for $\supp[1][2]$  in the sum \eqref{eq:U_part}.
Note that $\mathsf{U}^{(1,2)}_k$ is described as $\setsymbolin{\textcolor{black}{\bm{n}}}{\supp[1][2]}{n_k\ne(p-1)/2,\ \textcolor{black}{\left| \bm{n} \right|}\ne(3p-3)/2}.$
If $n_k=(p-1)/2$, then $1/2+n_k=p/2=0$ in $\mathbb{F}_p$, and if $\textcolor{black}{\left| \bm{n} \right|}=(3p-3)/2,$ then $n_1=n_2=n_3=(p-1)/2$, whence the sum does not change if we replace $\mathsf{U}^{(1,2)}_k$ for $\supp[1][2]$. 

\textcolor{black}{From the above consideration, simple computations imply that \eqref{eq:contiguity2} is
\[
\sum_{\textcolor{black}{\bm{n}\in}\Tilde{\mathsf{V}}^{(1,1)}}\frac{1}{4}\Tilde{A}_{\bm{n}}^{(1,1)}\bm{z}^{\bm{n}}
+\sum_{\textcolor{black}{\bm{n}\in}\supp[1][2]}
    \frac{1}{2}\cdot\frac{\frac{1}{2}+\textcolor{black}{\left| \bm{n} \right|}}{1+\left| \bm{n} \right|}
    \Tilde{A}_{\bm{n}}^{(1,2)}\bm{z}^{\bm{n}}
-\frac{1}{4}\sum_{\textcolor{black}{\bm{n}\in}\supp[1][1]}\Tilde{A}_{\bm{n}}^{(1,1)}\bm{z}^{\bm{n}}.
\]}
\textcolor{black}{We rewrite this using $\supp[1][1]=\Tilde{\mathsf{V}}^{(1,1)}\sqcup\supp[1][2]$ as}
\begin{align*}
\sum_{\textcolor{black}{\bm{n}\in}\supp[1][2]}
    \frac{1}{2}\cdot\frac{\frac{1}{2}+\textcolor{black}{\left| \bm{n} \right|}}{1+\textcolor{black}{\left| \bm{n} \right|}}
    \Tilde{A}_{\textcolor{black}{\bm{n}}}^{(1,2)}\textcolor{black}{\bm{z}^{\bm{n}}}
-\frac{1}{2}\cdot\frac{1}{2}\sum_{\textcolor{black}{\bm{n}\in}\supp[1][2]}\Tilde{A}_{\textcolor{black}{\bm{n}}}^{(1,1)}\textcolor{black}{\bm{z}^{\bm{n}}}.
\end{align*}

Then, by \eqref{eq:ALinterchanging}, the \textcolor{black}{above formula is zero}.

\underline{Case $i=2.$}
We note \textcolor{black}{$\mathsf{V}^{(2,2)}_k$ is}
\begin{align*}
&-\bm{e}_k+\setsymbolin{\textcolor{black}{\bm{n}}}{\supp[2][2]}{n_k=0\ \text{or}\ \textcolor{black}{\left| \bm{n} \right|}=0}\\
=&\left(-\bm{e}_k+\setsymbolin{\textcolor{black}{\bm{n}}}{\supp[2][2]}{n_k=0}\right)\cup\{\textcolor{black}{-\bm{e}_k}\},
\end{align*}
\textcolor{black}{which implies $\displaystyle \sum_{k=1}^3\sum_{\textcolor{black}{\bm{n}\in}\mathsf{V}^{(2,2)}_k}\left(n_k+1\right)\Tilde{A}_{\bm{n}+\bm{e}_k}^{(2,2)}\bm{z}^{\bm{n}}=0.$}
Next, we note
\[
\mathsf{U}^{(2,2)}_k
=\setsymbolin{\textcolor{black}{\bm{n}}}{\supp[2][2]}{n_k\ne(p-1)/2,\ \textcolor{black}{\left| \bm{n} \right|}\ne(p-1)/2}.
\]
From similar consideration to the case of $i=1,$ we can replace $\mathsf{U}^{(2,2)}_k$ for $\supp[2][2].$
\textcolor{black}{After some simple computations, we see that \eqref{eq:contiguity2} is equal to}
\[
\sum_{\textcolor{black}{\bm{n}\in}\supp[2][1]}
    \frac{1}{2}\cdot\frac{\frac{1}{2}+\textcolor{black}{\left| \bm{n} \right|}}{1+\textcolor{black}{\left| \bm{n} \right|}}
    \Tilde{A}_{\textcolor{black}{\bm{n}}}^{(2,2)}\textcolor{black}{\bm{z}^{\bm{n}}}
-\frac{1}{2}\cdot\frac{1}{2}\sum_{\textcolor{black}{\bm{n}\in}\supp[2][1]}\Tilde{A}_{\textcolor{black}{\bm{n}}}^{(2,1)}\textcolor{black}{\bm{z}^{\bm{n}}}.
\]
Then, by \eqref{eq:ALinterchanging}, the \textcolor{black}{above formula is zero.}
\end{proof}

There are other ``contiguity relations".
Here are some examples, see Remark \ref{rem:OriginalContiguityRelations} and Expectation \ref{expect:DeterminantLike} below.

\begin{Rem}\label{rem:OriginalContiguityRelations}
We have the following two relations, which are easily predicted from \cite[Theorem 3.8.1]{Matsumoto} and Theorem \ref{thm:PDEinPositiveChar}:
for $i=1,2$,
\begin{enumerate}
\item[(1)]
$\displaystyle
-\left( \sum_{k=1}^3 {(1-z_k)\partial_k} -1 \right) c_{ip-2}
	= \left( \sum_{k=1}^3 {z_k(1-z_k)\partial_k} -\sum_{k=1}^3 {\dfrac{1}{2}
	z_k} +\dfrac{1}{2} \right) c_{ip-1},
$
\item[(2)]
$\displaystyle
\left( \sum_{k=1}^3 {z_k\partial_k} +\dfrac{1}{2} \right) c_{ip-2}
	= \left( \sum_{k=1}^3 {z_k\partial_k} +1 \right) c_{ip-1}.
$
\end{enumerate}
Indeed (1) follows from Theorem \ref{thm:ContiguityRelation}.
Indeed
the left-hand side of (1) is equal to
$(c_{ip-2}-c_{ip-1})/2$ by Theorem \ref{thm:ContiguityRelation} (2), and the right-hand side of (1)
is equal to $(c_{ip-2}-c_{ip-1})/2$ by Theorem \ref{thm:ContiguityRelation} (1).
We can prove (2) in the similar way as in  Theorem \ref{thm:ContiguityRelation}, but we omit its proof, since we do not use it to prove Theorem \ref{thm:MultiplicityOne}.
\end{Rem}
Although a computational experiment finds the following beautiful relation, it has not been proven.
\begin{Expect}\label{expect:DeterminantLike}
 For $i=1,2,3$, we have
\[
\dfrac{\partial c_{p-1}} {\partial z_i}(\lambda)\dfrac{\partial c_{2p-2}} {\partial z_i}(\lambda)-\dfrac{\partial c_{2p-1}} {\partial z_i}(\lambda)\dfrac{\partial c_{p-2}} {\partial z_i}(\lambda)=0
\]   
for  $\lambda:=(\lambda_1,\lambda_2,\lambda_3)\in V(c_{p-1},c_{p-2},c_{2p-1},c_{2p-2})$.
\end{Expect}

\subsection{Proof of the multiplicity-one theorem}
Let $\lambda:=(\lambda_1,\lambda_2,\lambda_3)$ be a point of $V(c_{p-1},c_{p-2},c_{2p-1},c_{2p-2})$.
We are concerned about the rank of the Jacobian matrix at $\lambda$ 
\begin{equation}\label{JacobianMatrix}
J:=(v_{1,1},v_{2,1},v_{1,2},v_{2,2}):=
\begin{pmatrix}
\dfrac{\partial c_{p-1}} {\partial z_1}(\lambda) & \dfrac{\partial c_{2p-1}} {\partial z_1}(\lambda) & \dfrac{\partial c_{p-2}} {\partial z_1}(\lambda) & \dfrac{\partial c_{2p-2}} {\partial z_1}(\lambda)\\
\dfrac{\partial c_{p-1}} {\partial z_2}(\lambda) & \dfrac{\partial c_{2p-1}} {\partial z_2}(\lambda) & \dfrac{\partial c_{p-2}} {\partial z_2}(\lambda) & \dfrac{\partial c_{2p-2}} {\partial z_2}(\lambda)\\
\dfrac{\partial c_{p-1}} {\partial z_3}(\lambda) & \dfrac{\partial c_{2p-1}} {\partial z_3}(\lambda) & \dfrac{\partial c_{p-2}} {\partial z_3}(\lambda) & \dfrac{\partial c_{2p-2}} {\partial z_3}(\lambda)
\end{pmatrix}.
\end{equation}

\begin{Lem}\label{lem:same_j}
$v_{1,j}$ and $v_{2,j}$ are linearly independent at $\lambda$ for $j = 1,2$.
\end{Lem}
\begin{proof}
We consider the expansion of $c_{ip-j}$
at $\lambda:=(\lambda_1,\lambda_2,\lambda_3)$:
As the $z_\ell$-degree of $c_{ip-j}$ ($\ell = 1,2,3$) is at most $(p-1)/2$, 
we can write $c_{ip-j}$ as
\[
\sum_{k_1,k_2,k_3} \frac{1}{k_1!k_2!k_3!}\frac{\partial^{k_1+k_2+k_3}c_{ip-j}}{\partial z_1^{k_1}\partial z_2^{k_2}\partial z_3^{k_3}} (\lambda)
(z_1-\lambda_1)^{k_1}
(z_2-\lambda_2)^{k_2}
(z_3-\lambda_3)^{k_3}.
\]

If $v_{1,j}$ and $v_{2,j}$ were linearly dependent at $\lambda$, then
the expansion at $\lambda$ of $c_{p-j}$ is a scalar multiple of that of $c_{2p-j}$ up to degree one.
Since $c_{p-j}$ and $c_{2p-j}$ satisfy the same partial differential equations in Theorem \ref{thm:PDEinPositiveChar},
an inductive argument shows that 
$c_{p-j}$ is a scalar multiple of that of $c_{2p-j}$.
This is absurd, as the total degrees of $c_{p-j}$ and $c_{2p-j}$ are different.
\end{proof}
\begin{Thm}
The rank of $J$ is three.
\end{Thm}
\begin{proof}
We claim that 
all $2\times 2$ submatrices of $(v_{1,1},v_{2,1})$
and $(v_{1,2},v_{2,2})$ are of rank $2$.
This follows from Theorem \ref{thm:ContiguityRelation}.
Indeed, otherwise, for example \textcolor{black}{consider}
\begin{equation}\label{eg:rank1}
\textcolor{black}{A:=}\begin{pmatrix}
  \dfrac{\partial c_{p-2}} {\partial z_1}(\lambda) & \dfrac{\partial c_{2p-2}} {\partial z_1}(\lambda)\\
  \dfrac{\partial c_{p-2}} {\partial z_2}(\lambda) & \dfrac{\partial c_{2p-2}} {\partial z_2}(\lambda) 
\end{pmatrix}
\end{equation}
\textcolor{black}{and assume that $A$ were of} rank $\leq 1$.
Then, \textcolor{black}{since  
\begin{equation}\label{eq:5-23}
(1-\lambda_1)\dfrac{\partial c_{ip-2}} {\partial z_1}(\lambda)
+ (1-\lambda_2)\dfrac{\partial c_{ip-2}} {\partial z_2}(\lambda)
+ (1-\lambda_3)\dfrac{\partial c_{ip-2}} {\partial z_3}(\lambda)
=0
\end{equation}
for $i=1,2$ by Theorem \ref{thm:ContiguityRelation} (2), the vector} $\begin{pmatrix}\dfrac{\partial c_{p-2}} {\partial z_3}(\lambda) & \dfrac{\partial c_{2p-2}} {\partial z_3}(\lambda)\end{pmatrix}$ is a linear combination of
the two row vectors of \eqref{eg:rank1}
\textcolor{black}{by $\lambda_3\ne 1$ (Proposition \ref{prop:SspImpliesNonsing})}.
This means that $(v_{1,2},v_{2,2})$ is of rank $\leq 1$.
This contradicts Lemma \ref{lem:same_j}.
\textcolor{black}{Hence ${\rm rank}\, A = 2$ holds.}

Assume that the rank of $J$ were less than or equal to $2$.
Let $w_1, w_2$ and $w_3$ be the first, second and third row vectors of $J$. 
\textcolor{black}{By \eqref{eq:5-23}}, the latter two entries of
$(1-\lambda_1)w_1 + (1-\lambda_2)w_2 + (1-\lambda_3) w_3$
are zero, \textcolor{black}{i.e.,
\[
\begin{pmatrix}
w_1\\w_2\\(1-\lambda_1)w_1 + (1-\lambda_2)w_2 + (1-\lambda_3) w_3
\end{pmatrix}
= \left(
\begin{array}{@{}c c@{}}
\begin{matrix}
* & *\\
* & *
\end{matrix} & A \\
\begin{matrix}
* & *
\end{matrix} & \begin{matrix}
0 & 0
\end{matrix}
\end{array}
\right).
\]
By ${\rm rank}\, J \leq 2$ and  ${\rm rank}\, A = 2$, we get}
\begin{equation}\label{eq:5-24}
(1-\lambda_1)w_1 + (1-\lambda_2)w_2 + (1-\lambda_3) w_3 = 0.
\end{equation}
The same argument with Theorem \ref{thm:ContiguityRelation} (1) shows
\begin{equation}\label{eq:5-25}
(\lambda_1-\lambda_1^2)w_1 + (\lambda_2-\lambda_2^2)w_2 + (\lambda_3-\lambda_3^2) w_3 = 0.
\end{equation}
Eliminating $w_3$ by these equations, 
\textcolor{black}{i.e., considering $(\lambda_3-\lambda_3^2)\times$ \eqref{eq:5-24}$-(1-\lambda_3)\times$ \eqref{eq:5-25}},
we get that $w_1$ and $w_2$ are linearly dependent, where we used the fact that
\textcolor{black}{
\[
(\lambda_3-\lambda_3^2)(1-\lambda_i)-(1-\lambda_3)(\lambda_i-\lambda_i^2)
=(\lambda_3-\lambda_i)(1-\lambda_3)(1-\lambda_i)
\]
for $i=1,2$ are} not zero  by Proposition \ref{prop:SspImpliesNonsing}.
This contradicts \textcolor{black}{${\rm rank}\, A = 2$}.
\end{proof}

\end{document}